\documentclass[final,3p,times]{elsarticle}

\usepackage{amssymb}

\usepackage{amsfonts}
\usepackage{float}
\usepackage{comment}
\usepackage{amsmath}
\usepackage{color}
 \usepackage{amsthm}
  \usepackage{amsbsy}

 \newtheorem{theorem}{Theorem}[section]

 \newtheorem{lemma}[theorem]{Lemma}

\begin{document}


\begin{frontmatter}



\title{Total-coloring of planar graphs with maximum degree 6 and without prescribed 4-cycles}


\author[gzhu]{Enqiang Zhu\corref{correspondingauthor}}
\cortext[correspondingauthor]{Corresponding author.}
\ead{zhuenqiang@gzhu.edu.cn}
\author[dg]{Yangyang Zhou}
\author[dg]{Jin Xu}

\address[gzhu]{Institute of Computing Science and Technology, Guangzhou University, Guangzhou 510006, China}
\address[dg]{School of Electronics Engineering and Computer Science, Peking University, Beijing, 100871, China}


\begin{abstract}
The Total Coloring Conjecture (TCC) is a challenging unsolved problem posed by  Behzad   and Vizing  independently, which states that every simple graph $G$  admits a ($\Delta(G)$ +2)-total-coloring, where $\Delta(G)$ denotes the maximum degree of $G$.
This conjecture has been confirmed for graphs with $\Delta(G)\leq 5$. However, for planar graphs, the only open case is $\Delta(G)=6$. It was known that planar graphs with maximum degree 6 and without 4-cycles are 7-totally-colorable. In this paper, we improve this result by  showing that  any planar graph $G$ of maximum degree 6, which does not contain some special 4-cycles, is 7-totally-colorable.
\end{abstract}


\end{frontmatter}


\section{Introduction}

All graphs considered in this paper are  simple, finite and undirected, and  for the terminologies and notations not defined here we follow  \cite{Bondy2008,shao1}. For any graph $G$, we denote by $V(G),E(G),\Delta(G)$ and $\delta(G)$ the \emph{vertex set}, the \emph{edge set}, the \emph{maximum degree} and the \emph{minimum degree} of $G$, respectively. For any vertex $v$ in $G$, a vertex $u\in V(G)$ is said to be a neighbor of $v$ if $uv\in E(G)$. We use $N_G(v)$ to denote the set of neighbors of $v$.
The degree of $v$ in $G$, denoted by $d_G(v)$, is the number of neighbors of $v$ in $G$, i.e. $d_G(v)=|N_G(v)|$. A $k$-vertex, a $k^-$-vertex or a $k^+$-vertex is a vertex of degree $k$, at most $k$ or at least $k$. A $k$ ($k^-$ or $k^+$)-neighbor of a vertex $v$ is a neighbor  of $v$ with degree $k$ (at most $k$ or at least $k$). A \emph{$k$-cycle} is a cycle of length $k$. A chord of a cycle is an edge between two vertices of the cycle that is not an edge of the cycle. A cycle is called as chord-cycle if it has at least one chord. Obviously, if a 4-cycle is a chord-cycle, then it contains two adjacent 3-cycle, where two cycles are \emph{adjacent} if they share a common edge.



 A $k$-$total$-$coloring$ of a graph $G$ is a mapping $f$ from $V(G)\cup E(G)$ to the set of colors $\{1,2,\ldots,k\}$ such that $f(x)\neq f(y)$ for every pair of adjacent or incident elements $x,y\in V(G)\cup E(G)$. $G$ is $k$-$totally$-$colorable$ if it admits a $k$-total-coloring. The \emph{total chromatic number} of a graph $G$, denoted by  $\chi_t(G)$, is the smallest integer $k$ such that $G$ is $k$-totally-colorable. It is clear that $\chi_t(G)\geq \Delta(G)+1$. Behzad \cite{Behzad1965} and Vizing \cite{Vizing1968} independently conjectured that $\chi_t(G)\leq \Delta(G)+2$ for every simple graph $G$. This conjecture  has been confirmed for graphs with $\Delta\leq 5$ \cite{Kostochka1996}, and for planar graphs with $\Delta\geq 7$ \cite{Borodin1989} \cite{Jensen1995} \cite{Sanders1999}. Therefore, the only open case for planar graphs is $\Delta=6$. In 2009, 
 Lan Shen and Yingqian Wang \cite{Shen2009} proved that planar graphs with maximum degree 6 and without 4-cycles are 7-totally-colorable. In this paper, we consider planar graphs  that contain 4-cycles, and give a stronger statement:


\begin{theorem}\label{mainresult}
Let $G$ be  a planar graph of maximum degree  $6$. Then, $G$ has a 7-total-coloring if $G$ does not contain any 4-cycle $C$ such that  $C$ is a chord-cycle, or $C$ is adjacent to a 3-cycle, or  $C$ is incident with a 6-vertex $v$ and $C$ is  adjacent to a $5^-$-cycle that is incident with $v$.
\end{theorem}


The proof of this theorem runs as follows: We first explore some properties  of a minimal counterexample, and then we will obtain a contradiction through a discharging method.

\section{Reducible configurations}

Let $f$ be a $k$-total-coloring of a planar graph $G$. We use $\{1,2,\ldots,k\}$ to denote the color set with $k$ colors. For a vertex $v\in V(G)$, we use  $C_{f}(v)$ to denote the set of colors appearing on $v$ and the edges incident with $v$.
For a planar graph $G$, we always assume $G$ is embedded in the plane, and denote by $F(G)$ the set of faces of $G$. The degree of a face $f\in F(G)$, denoted by $d_G(f)$, is the number of edges incident with it, where each cut-edge is counted twice. A face of degree $k$ is called a \emph{$k$-face}. A $k$-face, denoted by $v_1v_2\ldots v_k$, with consecutive vertices $v_1,v_2,\ldots,v_k$ along its boundary in some direction is often said to be a $(d_G(v_1), d_G(v_2),\ldots,d_G(v_k))$-$face$.

Let $H$ be a  minimal counterexample to Theorem \ref{mainresult} in terms of the number of vertices and edges.  Then, $H$ has the following properties:

 (1) $H$ is a planar graph of maximum degree  6;

 (2) $H$ does not contain any 4-cycle $C$ such that $C$ is a chord-cycle, or $C$ is adjacent to a 3-cycle, or  $C$ is incident with a 6-vertex $v$ and $C$ is  adjacent to a $5^-$-cycle that is incident with $v$;

 (3) $H$ is not 7-totally-colorable.

Furthermore, since every planar graph with $\Delta\leq 5$ is totally 7-colorable \cite{Kostochka1996} and every subgraph of $H$ also satisfies the two conditions of Theorem \ref{mainresult}, it follows that every proper subgraph of $H$  is 7-totally-colorable. We first show some known results.

\begin{lemma}\label{lemma1}
 $(1)$ For every edge $uv\in E(H)$, if $d_H(u)\leq 3$ or $d_H(v)\leq 3$, then $d_H(u)+d_H(v)\geq 8$.

 $(2)$ The subgraph induced by all edges, whose two ends are 2-vertex and 6-vertex respectively in $H$ is a forest.

 $(3)$ $H$ has no (4,4,4)-face.
\end{lemma}
\begin{proof}
For (1), we suppose that $d_H(u)\leq 3$ and $d_H(u)+d_H(v)\leq 7$.  By the minimality of $H$, $H-\{uv\}$ has a 7-total-coloring $f$. Since $d_H(u)+d_H(v)\leq 7$, it follows that $d_{H-\{uv\}}(u)+d_{H-\{uv\}}(v)\leq 5$ and $|C_f(u)\cup C_f(v)|\leq 7$. Now, erase the color of $u$. Then, edge $uv$ can be colored with a color in $\{1,2,\ldots, k\}\setminus (C_f(v)\cup(C_f(u)\setminus \{f(u)\}))$. Therefore, we can obtain a 7-total-coloring of $H$ after coloring $u$ with an available color (since we assume that $d_H(u)\leq 3$),  a contradiction.

For (2), let $G'$ be the subgraph induced by all edges whose two ends are 2-vertex and 6-vertex, respectively. If $G'$ contains a cycle, say $C=u_1u_2\ldots u_ku_1$, then by (1)  $k \equiv 0$ (mod 2). Without loss of generality, we assume that $d_H(u_i)=2$ for $i\equiv 1$ (mod 2) and $d_H(u_j)=6$ for $j\equiv 0$ (mod 2), $i,j\in \{1,2,\ldots,k\}$. By the minimality of $H$, $H-E(C)$ has a 7-total-coloring. Erase the colors of $u_i$ for $i=1,2,\ldots, k$.  Then, each edge in $E(C)$ has two available colors. Therefore, we can obtain a 7-total-coloring of $H$, by first coloring edges in $E(C)$ properly (since every  cycle of even length is edge-2-choosable) and then coloring each 2-vertex $u_i$, $i=1,2,\ldots,k$, with an available color, a contradiction.

For (3), suppose that $H$ has a (4,4,4)-face $u_1u_2u_3$. By the minimality of $H$, $H-\{u_1u_2,u_2u_3,u_3u_1\}$ has a 7-total-coloring. Erase the colors of $u_1, u_2$ and $u_3$. Then, each element in $\{u_1,u_2,u_3,u_1u_2,u_2u_3,u_3u_1\}$ has at least three available colors. Since every 3-cycle is
totally 3-choosable, $H$ has a 7-total-coloring, and a contradiction. \qed
\end{proof}

Lemma \ref{lemma1} (1) indicates that $H$ does not contain any 1-vertex.
For any component $T$ of the forest stated in Lemma \ref{lemma1} (2),  it is easy to see that all leaves (i.e. 1-vertices) of $T$ are 6-vertices. Thus, $T$ has a maximum matching $M$ that saturates  every 2-vertex in $T$. For each 2-vertex $v$ in $T$, we refer to the neighbor of $v$ that is saturated by $M$ as the \emph{master} of $v$ (see \cite{Borodin1997}). Obviously, for a given $M$, each 6-vertex can be the master of at most one 2-vertex, and each 2-vertex has exactly one master.

\begin{lemma}\label{lemma3}
Every 4-face in $H$ is incident with at most one 2-vertex.
\end{lemma}

\begin{proof}
The result follows straightforwardly from Lemma \ref{lemma1} (1) and (2).  \qed
\end{proof}

\begin{lemma}\label{lemma2}
Let $v_1v_2v_3$ be a (2,6,6)-face of $H$, where $d_H(v_1)=2, d_H(v_2)=d_H(v_3)=6$. Then, $v_2$ (resp. $v_3$) has no 2-neighbor different from $v_1$.
\end{lemma}
\begin{proof}
We are sufficient to show that the result holds for $v_2$. Suppose that $N_H(v_2)\setminus \{v_1\}$ contains a 2-vertex, say $u$. By the minimality of $H$, $H-v_1v_2$ has a 7-total-coloring $f$. Erase the colors of $v_1$ and $u$.  Without loss of generality, we assume $C_f(v_2)$=$\{1,2,3,4,5,6\}$. If $f(v_1v_3)\neq 7$, then edge $v_1v_2$ can be properly colored by 7. Since $v_1,u$ are 2-vertices, there are at least three available colors for each of them. Hence, $H$ has a 7-total-coloring, and a contradiction. So we assume $f(v_1v_3)=7$. Denote by $w$ the other 6-neighbor (different from $v_2$) of $u$. When $f(uw)\neq 7$, we can  color $v_1v_2$ with $f(v_2u)$ and recolor $v_2u$ with 7. When $f(uw)=7$,  we can first interchange the colors of $v_1v_3$ and $v_2v_3$, and then color $v_1v_2$ with $f(v_2u)$ and recolor $v_2u$ with $f(v_2v_3)$. Clearly, based on the above discussion, we can obtain a 7-total-coloring of $H$ by coloring $v_1,u$ with two available colors. This contradicts the assumption of $H$.  \qed
\end{proof}

\begin{lemma}\label{lemma4}
$H$ has no (3,5,3,5)-face.
\end{lemma}
\begin{proof}
Suppose that $H$ has a (3,5,3,5)-face, say $q=v_1v_2v_3v_4$, where $d_H(v_1)=d_H(v_3)=3$,  $d_H(v_2)=d_H(v_4)=5$. By the minimality of $H$, $H-\{v_1v_2,v_2v_3,v_3v_4,v_4v_1\}$ has a 7-total-coloring $f$. Erase the colors of $v_1,v_3$. Then, each edge of $\{v_1v_2,v_2v_3,v_3v_4,$ $v_4v_1\}$ has at least two available colors. Since even cycles are 2-edge-choosable, we can properly color edges $v_1v_2,v_2v_3,v_3v_4$ and $v_4v_1$. Thus, after coloring $v_1$ and $v_3$ with two available colors, we obtain a 7-total-coloring of $H$, and a contradiction. Hence, $H$ does not contain any (3,5,3,5)-face.  \qed
\end{proof}

\section{Discharging}
We shall complete the proof of Theorem \ref{mainresult} by  using discharging procedure to derive a contradiction.  By Euler's formula $|V(H)|-|E(H)|+|F(H)|=2$, we have

\[
\sum\limits_{v\in V(H)}(d_H(v)-4)+\sum\limits_{f\in F(H)}(d_H(f)-4)=-8<0.
\]

We define \emph{ch} to be the \emph{initial charge}. Let $ch(v)=d_H(v)-4$ for each $v\in V(H)$ and $ch(f)=d_H(f)-4$ for each $f\in F(H)$. It follows that $\sum_{x\in (V(H)\cup F(H))}ch (x)<0$. Now, we will reassign a new charge denoted by  $ch'(x)$ to each $x\in (V(H)\cup F(H))$ according to the discharging rules below. If we can show that $ch'(x)\geq 0$ for each $x\in (V(H)\cup F(H))$, then we obtain a contradiction, and complete the proof.

Our discharging rules are as follows.


(r1) \emph{Charge to a 2-vertex $v$}

(r1.1) When $v$ is incident with a 3-face $T$, $v$ gets $\frac{1}{2}$ from each of its 6-neighbors, and receives 1 from its incident $5^+$-face.

(r1.2) When $v$ is not incident with any 3-face, $v$ receives 1 from its master, and gets $\frac{1}{2}$ from each of its incident $4^+$-face.

(r2) \emph{Charge to a 3-vertex $v$}

(r2.1) When $v$ is incident with a 3-face, it receives $\frac{1}{2}$ from each of its incident $5^+$-faces.

(r2.2) When $v$ is not incident with any 3-face, it gets $\frac{1}{3}$ from each of its incident $4^+$-faces.

(r3) \emph{Charge to a 3-face $f$}

(r3.1) When $f$ is a ($5^+$,$5^+$,$5^+$), it receives $\frac{1}{3}$ from each of its incident $5^+$-vertices.

(r3.2) When $f$ is  incident with exactly one $4^-$-vertex, it receives $\frac{1}{2}$ from each of its incident $5^+$-vertices.

(r3.3) When $f$ is  incident with exactly two $4^-$-vertices, it receives $\frac{1}{2}$ from the remaining incident $5^+$-vertex, and gets $\frac{1}{6}$ from each of its adjacent $5^+$-faces.

(r4) \emph{Charge to a 4-face $f$}

(r4.1)  $f$ is incident with a 2-vertex: (1)If $f$ is not incident with any 3-vertex, then it receives $\frac{1}{4}$ from each of its incident $6$-vertices, and $\frac{1}{5}$ from each of its incident $5$-vertices. (2) If $f$ is incident with a 3-vertex, then it receives $\frac{5}{12}$ from each of its incident $6$-vertices. 

(r4.2) $f$ is not incident with any 2-vertex but is incident with a 3-vertex: (1) If $f$ is incident with only one 3-vertex,  then when $f$ is a (3,6,4,6), $f$ receives  $\frac{1}{6}$  from each of its incident $6$-vertices; when $f$ is not a (3,6,4,6)-face, $f$ receives $\frac{1}{5}$ from each of its incident $5$-vertices, and $\frac{2}{15}$  from each of its incident $6$-vertices. (2)
If $f$ is incident with two 3-vertices, then when $f$ is incident with  two  6-vertices, it receives $\frac{1}{3}$ from each of its incident $6$-vertices; when $f$ is incident with one 6-vertex, it receives $\frac{7}{15}$ from  its  $6$-neighbor, and $\frac{1}{5}$ from its $5$-neighbor.

(r5) Every $5^+$-face with positive charge after the above distribution, r1 to r4, distributes its remaining charges evenly among its incident 6-vertices.


The rest of this paper is to check that  $ch'(x)\geq 0$ for all $x\in (V(H)\cup F(H))$ which will be the desired contradiction. 

\begin{lemma} \label{chargeto6-vertex1}
Let $F$ be a 6-face that is adjacent to two 4-faces $S_1$ and $S_2$. If $F, S_1$ and  $S_2$ are incident with a common 6-vertex $v$ (see Figure \ref{6-vertex} (a)), then

$(1)$ when $d_H(v_1)=d_H(v_2)=2$, $v$ receives at least $\frac{1}{6}$ from $F$;

$(2)$ when $d_H(v_1)=d_H(v_2)=3$, $v$ receives at least $\frac{5}{18}$ from $F$;

$(3)$ when $\{d_H(v_1), d_H(v_2)\}=\{2,3\}$, $v$ receives at least $\frac{2}{9}$ from $F$;

$(4)$ when $\{d_H(v_1), d_H(v_2)\}=\{2,4^+\}$, $v$ receives at least $\frac{1}{4}$ from $F$;

$(5)$ when $\{d_H(v_1), d_H(v_2)\}=\{3,4^+\}$, $v$ receives at least $\frac{7}{24}$ from $F$;

$(6)$ when $\{d_H(v_1), d_H(v_2)\}=\{4^+,4^+\}$, $v$ receives at least $\frac{1}{4}$ from $F$.

\end{lemma}
\begin{proof}
Obviously, if $d_H(v_1)\leq 3$ and $d_H(v_2)\leq 3$, then $v_1,v_2$ are not incident with any 3-face. Note that if $F$ is incident with a 2-vertex $w$, then   $w$ is not incident with a 3-face; otherwise the subgraph induced by vertices in $F$ contains a 5-cycle, which is adjacent to a 4-cycle $S_1$ or $S_2$, and a contradiction. Therefore, the $3^-$-vertex incident with $F$ can receive at most $\frac{1}{2}$ from $F$ by r1 and r2.

\begin{figure}[H]
\centering
  \includegraphics[width=75pt]{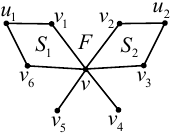}\hspace{0.5cm}
  \includegraphics[width=65pt]{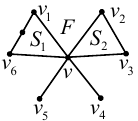}\hspace{0.5cm}
    \includegraphics[width=80pt]{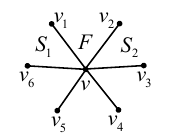}\\
  (a)\hspace{2.5cm} (b)\hspace{2.5cm} (c)
  \caption{Illustrations for Lemmas \ref{chargeto6-vertex1}, \ref{chargeto6-vertex2},\ref{chargeto6-vertex}}\label{6-vertex}
\end{figure}

For (1), it has that $d_H(u_1)=d_H(u_2)=6$, and both $u_1$ and $u_2$ are incident with $F$. Let $F=u_1v_1vv_2u_2w$. By Lemma \ref{lemma1} (2), we have  $d_H(w)\geq 3$. If $d_H(w)=3$, then $F$ is not adjacent to any ($4^-,4^-,5^+$)-face. Therefore, $F$ has at least $6-4-\frac{1}{2}-\frac{1}{2}-\frac{1}{2}=\frac{1}{2}$ after r1 to r4, and $F$ gives $v$ at least $\frac{1}{3}\times \frac{1}{2}=\frac{1}{6}$ by r5. If $d_H(w)=4$, then $F$ is adjacent to at most two ($4^-,4^-,5^+$)-faces. Therefore, $F$ has at least $6-4-\frac{1}{2}-\frac{1}{2}-2\times\frac{1}{6}=\frac{2}{3}$ after r1 to r4, and $F$ gives $v$ at least $\frac{2}{3}\times \frac{1}{3}=\frac{2}{9}$ by r5. If $d_H(w)\geq 5$, then $F$ is not adjacent to any ($4^-,4^-,5^+$)-face. Therefore, $F$ has at least $6-4-\frac{1}{2}-\frac{1}{2}=1$ after r1 to r4, and $F$ gives $v$ at least $1\times \frac{1}{4}=\frac{1}{4}$ by r5.

For (2), let $F=v_1vv_2w_1w_2w_3$. By Lemma \ref{lemma1} (1), we have  $d_H(w_1)\geq 5$, $d_H(w_3)\geq 5$. If $d_H(w_2)\leq 3$, then $F$ is not adjacent to any ($4^-,4^-,5^+$)-face. Therefore, $F$ has at least $6-4-\frac{1}{3}-\frac{1}{3}-\frac{1}{2}=\frac{5}{6}$ after r1 to r4, and $F$ gives $v$ at least $\frac{5}{6}\times \frac{1}{3}=\frac{5}{18}$ by r5. If $d_H(w_2)=4$, then $F$ is adjacent to at most two ($4^-,4^-,5^+$)-faces. Therefore, $F$ has at least $6-4-\frac{1}{3}-\frac{1}{3}-2\times\frac{1}{6}=1$ after r1 to r4, and $F$ gives $v$ at least $1\times \frac{1}{3}=\frac{1}{3}$ by r5. If $d_H(w_2)\geq 5$, then $F$ is not adjacent to any ($4^-,4^-,5^+$)-face. Therefore, $F$ has at least $6-4-\frac{1}{3}-\frac{1}{3}=\frac{4}{3}$ after r1 to r4, and $F$ gives $v$ at least $\frac{4}{3}\times \frac{1}{4}=\frac{1}{3}$ by r5.

For (3), without loss of generality, we assume $d_H(v_1)=2, d_H(v_2)=3$. Then, $u_1$ is incident with $F$. Let $F=u_1v_1vv_2w_1w_2$. By Lemma \ref{lemma1} (1), it has that  $d_H(w_1)\geq 5$, $d_H(u_1)=6$. If $d_H(w_2)\leq 3$, then $F$ is not adjacent to any ($4^-,4^-,5^+$)-face. Therefore, $F$ has at least $6-4-\frac{1}{3}-\frac{1}{2}-\frac{1}{2}=\frac{2}{3}$ after r1 to r4, and $F$ gives $v$ at least $\frac{2}{3}\times \frac{1}{3}=\frac{2}{9}$ by r5. If $d_H(w_2)=4$, then $F$ is adjacent to at most two ($4^-,4^-,5^+$)-faces. Therefore, $F$ has at least $6-4-\frac{1}{3}-\frac{1}{2}-2\times\frac{1}{6}=\frac{5}{6}$ after r1 to r4, and $F$ gives $v$ at least $\frac{5}{6}\times \frac{1}{3}=\frac{5}{18}$ by r5. If $d_H(w_2)\geq 5$, then $F$ is not adjacent to any ($4^-,4^-,5^+$)-face. Therefore, $F$ has at least $6-4-\frac{1}{3}-\frac{1}{2}=\frac{7}{6}$ after r1 to r4, and $F$ gives $v$ at least $\frac{7}{6}\times \frac{1}{4}=\frac{7}{24}$ by r5.

For (4), without loss of generality, we assume $d_H(v_1)=2, d_H(v_2)\geq 4$. Then, $u_1$ is incident with $F$. Let $F=u_1v_1vv_2w_1w_2$. By Lemma \ref{lemma1} (1), it has that  $d_H(u_1)=6$. If $\{w_1,w_2\}$ contains a $3^-$-vertex, then $F$ is  adjacent to at most one ($4^-,4^-,5^+$)-face when $d_H(v_2)=4$, and $F$ is not adjacent to any ($4^-,4^-,5^+$)-face when $d_H(v_2)\geq 5$.  Therefore, $F$ gives $v$ at least $\min$\{$(6-4-\frac{1}{6}-\frac{1}{2}-\frac{1}{2})\times \frac{1}{3}$, $(6-4-\frac{1}{2}-\frac{1}{2})\times \frac{1}{4}$\}=$\frac{1}{4}$.  If $d_H(w_1)\geq 4$ and $d_H(w_2)\geq 4$, then $F$ is  adjacent to at most three ($4^-,4^-,5^+$)-faces. Therefore, $F$ gives $v$ at least $\min$\{ $(6-4-\frac{1}{2})\times \frac{1}{5}$, $(6-4-\frac{1}{2}-\frac{1}{6})\times \frac{1}{4}$, $(6-4-\frac{1}{2}-2\times \frac{1}{6})\times \frac{1}{4}$, $(6-4-\frac{1}{2}-3\times \frac{1}{6})\times \frac{1}{3}$\}=$\frac{7}{24}$.

For (5), without loss of generality, we assume $d_H(v_1)=3, d_H(v_2)\geq 4$. Let $F=v_1vv_2w_1w_2w_3$. By Lemma \ref{lemma1} (1), it has that  $d_H(w_3)\geq 5$.
 If $\{w_1,w_2\}$ contains a $3^-$-vertex, then $F$ is  adjacent to at most one ($4^-,4^-,5^+$)-face when $d_H(v_2)=4$, and $F$ is not adjacent to any($4^-,4^-,5^+$)-face when $d_H(v_2)\geq 5$.  Therefore, $F$ gives $v$ at least $\min$\{$(6-4-\frac{1}{6}-\frac{1}{3}-\frac{1}{2})\times \frac{1}{3}$, $(6-4-\frac{1}{3}-\frac{1}{2})\times \frac{1}{4}$\}=$\frac{7}{24}$.  If $d_H(w_1)\geq 4$ and $d_H(w_2)\geq 4$, then $F$ is  adjacent to at most three ($4^-,4^-,5^+$)-faces. Therefore, $F$ gives $v$ at least $\min$\{$(6-4-\frac{1}{3})\times \frac{1}{5}$, $(6-4-\frac{1}{3}-\frac{1}{6})\times \frac{1}{4}$, $(6-4-\frac{1}{3}-2\times \frac{1}{6})\times \frac{1}{4}$, $(6-4-\frac{1}{3}-3\times \frac{1}{6})\times \frac{1}{3}$\}=$\frac{1}{3}$.

For (6),  let $F=v_1vv_2w_1w_2w_3$. By Lemma \ref{lemma1} (1), $\{w_1,w_2,w_3\}$ contains at most two $3^-$-vertices. If $\{w_1,w_2,w_3\}$ contains exactly two $3^-$-vertices, then $F$ is not adjacent to any($4^-,4^-,5^+$)-face, and  $d_H(w_1)\leq 3$, $d_H(w_3)\leq 3$ and $d_H(w_2)\geq 5$.  Therefore, $F$ has at least $6-4-\frac{1}{2}-\frac{1}{2}=1$ after r1 to r4, and $F$ gives $v$ at least $1\times \frac{1}{4}=\frac{1}{4}$ by r5. If $\{w_1,w_2,w_3\}$ contains  one $3^-$-vertex, then $F$ is adjacent to at most two ($4^-,4^-,5^+$)-faces. Therefore, $F$ gives $v$ at least $\min$\{$(6-4-\frac{1}{2})\times \frac{1}{5}$, $(6-4-\frac{1}{2}-2\times \frac{1}{6})\times \frac{1}{4}$\}=$\frac{7}{24}$. If $\{w_1,w_2,w_3\}$ contains  no $3^-$-vertex, then $F$ is adjacent to at most four ($4^-,4^-,5^+$)-faces. Therefore, $F$ gives $v$ at least $\min$\{$(6-4)\times \frac{1}{6}$, $(6-4-2\times \frac{1}{6})\times \frac{1}{5}$, $(6-4-4\times \frac{1}{6})\times \frac{1}{4}$\}=$\frac{1}{3}$.  \qed
\end{proof}

\begin{lemma} \label{chargeto6-vertex2}
Let $F$ be a 6-face that is adjacent to a 4-face $S_1$ and a 3-face $S_2$. If $F, S_1, S_2$ are incident with a common 6-vertex $v$ (see Figure \ref{6-vertex} (b)), then

$(1)$ when $d_H(v_1)=2, d_H(v_2)=3$, $v$ receives at least $\frac{1}{6}$ from $F$;

$(2)$ when $d_H(v_1)=2, d_H(v_2)\geq4$, $v$ receives at least $\frac{2}{9}$ from $F$;

$(3)$ when $d_H(v_1)= d_H(v_2)=3$, $v$ receives at least $\frac{2}{9}$ from $F$;

$(4)$ when $d_H(v_1)=3, d_H(v_2)\geq 4$, $v$ receives at least $\frac{5}{18}$ from $F$;

$(5)$ when $d_H(v_1)\geq 4, d_H(v_2)\geq 4$, $v$ receives at least $\frac{1}{4}$ from $F$.
\end{lemma}
\begin{proof}
For (1), $F$ is incident with at most two 2-vertices. In particular, when  $F$ is incident with exactly two 2-vertices, no such 2-vertex is incident with a 3-face by Lemma \ref{lemma2}. Therefore, $F$ give $v$ at least $(6-4-3\times\frac{1}{2})\times \frac{1}{3}=\frac{1}{6}$.

For (2), $F$ is also incident with at most two 2-vertices, and there does not exist  such 2-vertex to be incident with a 3-face by  the condition (2) of Theorem \ref{mainresult}.  Therefore, $F$ give $v$ at least $(6-4-\frac{1}{2}-\frac{1}{2}-2\times \frac{1}{6})\times \frac{1}{3}=\frac{2}{9}$ (when $v_2$ is a 4-vertex).

With an analogous proof, we have

For (3),  $F$ give $v$ at least $(6-4-\frac{1}{2}-\frac{1}{2}-\frac{1}{3})\times \frac{1}{3}=\frac{2}{9}$ (when $F$ is incident with a 2-vertex, or two 3-vertices incident with  3-faces).

For (4),  $F$ give $v$ at least $(6-4-\frac{1}{2}-\frac{1}{3}-2\times \frac{1}{6})\times \frac{1}{3}=\frac{5}{18}$ (when $v_2$ is a 4-vertex and $F$ is adjacent to two ($4^-,4^-,5^+$)-faces).

For (5),  $F$ give $v$ at least $(6-4-\frac{1}{2}-\frac{1}{2})\times \frac{1}{4}=\frac{1}{4}$ (when $F$ is incident with two $3^-$-vertices).  \qed
\end{proof}

\begin{lemma} \label{chargeto6-vertex}
Let $F$ be a $7^+$-face that is adjacent to two
$4^-$-faces $S_1$ and $S_2$, and $F$, $S_1$ and $S_2$ are incident with a common 6-vertex $v$, see Figure \ref{6-vertex} (c).  If $v_1$ (resp. $v_2$) is not incident with a 3-face when $v_1$ (resp. $v_2$) is a 2-vertex, then





$(1)$ when $d_H(v_1)=d_H(v_2)=2$, $v$ receives at least $\frac{1}{4}$ from $F$;

$(2)$ when $\{d_H(v_1), d_H(v_2)\}=\{2,3\}$, $v$ receives at least $\frac{1}{4}$ from $F$;

$(3)$ when $d_H(v_1)=d_H(v_2)=3$, $v$ receives at least $\frac{1}{4}$ from $F$;

$(4)$ when $\{d_H(v_1), d_H(v_2)\}=\{2,4^+\}$, $v$ receives at least $\frac{1}{4}$ from $F$;

$(5)$ when $\{d_H(v_1), d_H(v_2)\}=\{3,4^+\}$, $v$ receives at least $\frac{1}{8}$ from $F$;

$(6)$ when $\{d_H(v_1), d_H(v_2)\}=\{4^+,4^+\}$, $v$ receives at least $\frac{1}{6}$ from $F$.
\end{lemma}
\begin{proof}
Let $F$ be a $k$ ($\geq 7$)-face. Denote by $\ell_2$ and $\ell_3$  the number of $2$-vertices and $3$-vertices incident with $F$, respectively.  Let $\ell$ be the number of ($4^-,4^-,5^+$)-faces incident with $F$. Then,  $\ell\leq k-2(\ell_2+\ell_3)$, $\ell_2\leq \lceil\frac{k-2}{2}\rceil$ and $\ell_2+\ell_3\leq \lfloor\frac{k}{2}\rfloor$, and $F$ is incident with at least $\lceil\frac{\ell}{2}\rceil$ 4-vertices.
 Note that when $k=7$, $F$ is incident with at most two 2-vertices that are incident with a 3-face. Otherwise, there exist a 4-cycle that is adjacent to a 3-cycle. What's more, if $d_H(v_i)\leq 3$ for $i=1,2$, then $v_i$ is not incident with any 3-face except $S_1$ and $S_2$.

(1) $\{d_H(v_1), d_H(v_2)\}=\{2\}$. Then neither $v_2$ nor $v_3$ is incident with a 3-face by Lemma \ref{lemma2}. So, After r1 to r4, $F$ has at least $k-4-\ell\times \frac{1}{6}-(\ell_2-2)-2\times \frac{1}{2}-\ell_3\times \frac{1}{2}$. Therefore, $F$ gives $v$ at least
$\frac{k-4-\ell\times \frac{1}{6}-(\ell_2-2)-2\times \frac{1}{2}-\ell_3\times \frac{1}{2}}{k-\ell_2-\ell_3-\lceil\frac{\ell}{2}\rceil}$ $=$
$1+\frac{\frac{\ell_3}{2}+ \lceil\frac{\ell}{2}\rceil-\frac{\ell}{6}-3}{k-\ell_2-\ell_3-\lceil\frac{\ell}{2}\rceil}$ $\geq$ $1-\frac{3}{k-\ell_2}\geq \frac{1}{4}$.

(2) $\{d_H(v_1), d_H(v_2)\}=\{2,3\}$. Then, $\ell_3\geq 1$ and $\ell_2\leq \lfloor\frac{k-2}{2}\rfloor$, and $F$ has at least $k-4-\ell\times \frac{1}{6}-(\ell_2-1)- \frac{1}{2}-\ell_3\times \frac{1}{2}$  after r1 to r4. Therefore, $F$ gives $v$ at least
$\frac{k-4-\ell\times \frac{1}{6}-(\ell_2-1)-\frac{1}{2}-\ell_3\times \frac{1}{2}}{k-\ell_2-\ell_3-\lceil\frac{\ell}{2}\rceil}$  $=$
$1+\frac{\frac{\ell_3}{2}+ \lceil\frac{\ell}{2}\rceil-\frac{\ell}{6}-\frac{7}{2}}{k-\ell_2-\ell_3-\lceil\frac{\ell}{2}\rceil}$   $\geq$ $1+\frac{\frac{1}{2}-\frac{7}{2}}{k-\ell_2-1}\geq \frac{1}{4}$.

(3) $\{d_H(v_1), d_H(v_2)\}=\{3\}$. Then,  $\ell_3\geq 2$ and $\ell_2\leq \lfloor\frac{k-4}{2}\rfloor$, and $F$ has at least $k-4-\ell\times \frac{1}{6}-\ell_2-\ell_3\times \frac{1}{2}$  after r1 to r4. Therefore, $F$ gives $v$ at least
$\frac{k-4-\ell\times \frac{1}{6}-\ell_2-\ell_3\times \frac{1}{2}}{k-\ell_2-\ell_3-\lceil\frac{\ell}{2}\rceil}$ $=$
$1+\frac{\frac{\ell_3}{2}+ \lceil\frac{\ell}{2}\rceil-\frac{\ell}{6}-4}{k-\ell_2-\ell_3-\lceil\frac{\ell}{2}\rceil}$  $\geq$ $1+\frac{\frac{2}{2}-4}{k-\ell_2-2}\geq \frac{1}{4}$.

(4) $\{d_H(v_1), d_H(v_2)\}=\{2,4^+\}$. Without loss of generality, we assume $d_H(v_1)=2$ and $d_H(v_2)=4^+$. Then, $v_1$ is not incident with a 3-face. First, if $F$ is a 7-face, then $F$ can incident with at most three 2-vertices. In this case,  by Lemma \ref{lemma2} $F$ is incident with at most one 2-vertex that is incident with a 3-face. Therefore, $F$ gives $v$ at least $(7-4-\frac{1}{2}-1-\frac{1}{2})\times \frac{1}{4}=\frac{1}{4}$. Second, if $F$ is a $8^+$-face, then $\ell_2+\ell_3\leq \lceil\frac{k-2}{2}\rceil$, and $F$ has at least $k-4-\ell\times \frac{1}{6}-(\ell_2-1) -\frac{1}{2}-\ell_3\times \frac{1}{2}$  after r1 to r4. Therefore, $F$ gives $v$ at least
$\frac{k-4-\ell\times \frac{1}{6}-(\ell_2-1) -\frac{1}{2}-\ell_3\times \frac{1}{2}}{k-\ell_2-\ell_3-\lceil\frac{\ell}{2}\rceil}$ $=$
$1+\frac{\frac{\ell_3}{2}+ \lceil\frac{\ell}{2}\rceil-\frac{\ell}{6}-\frac{7}{2}}{k-\ell_2-\ell_3-\lceil\frac{\ell}{2}\rceil}$ $\geq$ $\min \{1-\frac{\frac{7}{2}+2\times \frac{1}{6}-1}{k-\ell_2-1}, 1-\frac{\frac{7}{2}}{k-\ell_2}\}\geq \min \{\frac{7}{24}, \frac{3}{10}\}$=$\frac{7}{24}$.

(5) $\{d_H(v_1), d_H(v_2)\}=\{3,4^+\}$. Then,  $\ell_3\geq 1$ and $\ell_2\leq \lceil\frac{k-4}{2}\rceil$, and $F$ has at least $k-4-\ell\times \frac{1}{6}-\ell_2-\ell_3\times \frac{1}{2}$  after r1 to r4. Therefore, $F$ gives $v$ at least
$\frac{k-4-\ell\times \frac{1}{6}-\ell_2-\ell_3\times \frac{1}{2}}{k-\ell_2-\ell_3-\lceil\frac{\ell}{2}\rceil}$  $=$
$1+\frac{\frac{\ell_3}{2}+ \lceil\frac{\ell}{2}\rceil-\frac{\ell}{6}-4}{k-\ell_2-\ell_3-\lceil\frac{\ell}{2}\rceil}$  $\geq$ $1+\frac{\frac{1}{2}-4}{k-\ell_2-1}\geq \frac{1}{8}$.

(6) $\{d_H(v_1), d_H(v_2)\}=\{4^+,4^+\}$. First, if $\{d_H(v_1), d_H(v_2)\}=\{4,4\}$, then $\ell_2+\ell_3\leq \lceil\frac{k-5}{2}\rceil$, and $F$ has at least $k-4-\ell\times \frac{1}{6}-\ell_2-\ell_3\times \frac{1}{2}$  after r1 to r4. Clearly, in this case each of $v_1$ and $v_2$ can be incident with two $(4^-,4^-,5^+)$-faces that are adjacent to $F$.  When $F$ is a 7-face, it has that $F$ is incident with at most one $3^-$-vertices. Therefore, $F$ gives $v$ at least $(7-4-1-4\times \frac{1}{6})\times \frac{1}{4}=\frac{1}{3}$. When $F$ is a 8-face, it has that $F$ is incident with at most two $3^-$-vertices. Particularly, when $F$ is incident with exactly two $2$-vertices, we have that neither of these two 2-vertices is incident with a 3-face by Lemma \ref{lemma2}. Therefore, $F$ gives $v$ at least $(8-4-1-\frac{1}{2}-4\times \frac{1}{6})\times \frac{1}{4}=\frac{11}{24}$ ($F$ is incident with a 2-vertex and a 3-vertex). When $F$ is a $9^+$-face,  $F$ gives $v$ at least
$\frac{k-4-\ell\times \frac{1}{6}-\ell_2-\ell_3\times \frac{1}{2}}{k-\ell_2-\ell_3-\lceil\frac{\ell}{2}\rceil}$ $\geq$ $\frac{k-4-4\times \frac{1}{6}-\lceil\frac{k-5}{2}\rceil}{k-2-\lceil\frac{k-5}{2}\rceil}$ $\geq \frac{7}{15}$ ($k=9$).
Second, if $\{d_H(v_1), d_H(v_2)\}=\{4,5^+\}$, say $d_H(v_1)=4$, then $v_1$ can be incident with two $(4^-,4^-,5^+)$-faces, and $\ell_2+\ell_3\leq \lceil\frac{k-4}{2}\rceil$. When $F$ is a 7-face, we have that $F$ is incident with at most two $3^-$-vertices. Particularly, when $F$ is incident with exactly two $2$-vertices, it has that no such 2-vertices is incident with a 3-face by Lemma \ref{lemma2}. Therefore, $F$ gives $v$ at least $(7-4-1-\frac{1}{2}-2\times\frac{1}{6})\times \frac{1}{4}=\frac{7}{24}$. When $F$ is a $8^+$-face, we have that $F$ has at least $k-4-\ell\times \frac{1}{6}-\ell_2-\ell_3\times \frac{1}{2}$  after r1 to r4, and $F$ gives $v$ at least
$\frac{k-4-\ell\times \frac{1}{6}-\ell_2-\ell_3\times \frac{1}{2}}{k-\ell_2-\ell_3-\lceil\frac{\ell}{2}\rceil}$$=$
$1+\frac{\frac{\ell_3}{2}+ \lceil\frac{\ell}{2}\rceil-\frac{\ell}{6}-4}{k-\ell_2-\ell_3-\lceil\frac{\ell}{2}\rceil}$  $\geq$
$1+\frac{\lceil\frac{2}{2}\rceil-\frac{2}{6}-4}{k-\lceil\frac{k-4}{2}\rceil-\lceil\frac{2}{2}\rceil} $=$\frac{1}{3}$ ($k=8$ and $\ell=2$). Third, if $\{d_H(v_1), d_H(v_2)\}=\{5^+,5^+\}$, then $\ell_2\leq \lceil\frac{k-3}{2}\rceil$, and $F$ has at least $k-4-\ell\times \frac{1}{6}-\ell_2-\ell_3\times \frac{1}{2}$  after r1 to r4. Therefore, $F$ gives $v$ at least
$\frac{k-4-\ell\times \frac{1}{6}-\ell_2-\ell_3\times \frac{1}{2}}{k-\ell_2-\ell_3-\lceil\frac{\ell}{2}\rceil}$ $=$
$1+\frac{\frac{\ell_3}{2}+ \lceil\frac{\ell}{2}\rceil-\frac{\ell}{6}-4}{k-\ell_2-\ell_3-\lceil\frac{\ell}{2}\rceil}$
$\geq \frac{1}{5}$.  \qed
\end{proof}

Let $F$ be a $6^+$-face that is adjacent to a $4$-face $S_1$ and a $4^-$-face $S_2$, and $F$, $S_1$ and $S_2$ are incident with a common 6-vertex $v$. If $S_1$ is not incident with a $3^-$-vertex that is adjacent to $v$, then  $F$ gives $v$ at least $\frac{1}{6}$ according to Lemmas \ref{chargeto6-vertex1}, \ref{chargeto6-vertex2}  and  \ref{chargeto6-vertex}.

\begin{lemma} \label{6faces}
Let $S$ be a $k(\geq5)$-face of $H$. Then, $S$ has nonnegative charges after the discharging procedures r1 to r4.
\end{lemma}
\begin{proof}
When $k\geq 6$, the conclusion hold by a similar proof as those in Lemmas \ref{chargeto6-vertex1}, \ref{chargeto6-vertex2} and \ref{chargeto6-vertex}. In the following, we consider the case of $k=5$.

Case 1. $S$ is not incident with any 2-vertex. If $S$ is not incident with a 3-vertex, then $S$ is adjacent to at most five ($4^-,4^-,5^+$)-faces. Hence, $S$  has at least $5-4-5\times \frac{1}{6}=\frac{1}{6}$ charges after r1 to r4. If  $S$ is incident with one 3-vertex, then $S$ is adjacent to at most three ($4^-,4^-,5^+$)-faces. So, $S$ has at least $5-4-\frac{1}{2}-3\times \frac{1}{6}$=$0$ charges after r1 to r4. If  $S$ is incident with two 3-vertices, then $S$ is not adjacent to any ($4^-,4^-,5^+$)-face. So, $S$ has at least $5-4-2\times \frac{1}{2}$=$0$ charges after r1 to r4. 

Case 2. $S$ is  incident with exactly one 2-vertex $x$. Obviously, $x$ is not incident with a 3-face. Otherwise, $H$ contains a 4-cycle that is adjacent to a 3-cycle, and a contradiction. If $S$ is not incident with a 3-vertex, then $S$ is adjacent to at most three ($4^-,4^-,5^+$)-faces. So,  $S$ has at least $5-4-\frac{1}{2}-3\times \frac{1}{6}$=$0$ charges after r1 to r4.  If  $S$ is incident with one 3-vertex  (note that $S$ is incident with at most one 3-vertex in this case), then $S$ is not adjacent to any ($4^-,4^-,5^+$)-face. Therefore,  $S$ has at least $5-4-\frac{1}{2}-\frac{1}{2}=0$ charges after r1 to r4.

Case 3. $S$ is  incident with exactly two 2-vertex  (note that $S$ is incident with at most two 2-vertices). Obviously, $S$ is not incident with any 3-vertex, and   by Lemma \ref{lemma1} (1)   $S$ is not adjacent to any ($4^-,4^-,5^+$)-face. Therefore, $S$ has at least $5-4-\frac{1}{2}-\frac{1}{2}=0$ charges after r1 to r4.  \qed
\end{proof}

\vspace{0.2cm}

\subsection{Final charge of vertices}
Let $v\in V(H)$ be a vertex of $H$. Obviously, $d_H(v)\geq 2$ by Lemma \ref{lemma1} (1).

Let $v$ be a 2-vertex. Then $v$ has two 6-neighbors by Lemma \ref{lemma1} (1). Clearly, by r1 we have $ch'(v)$=-2+2=0.

Let $v$ be a 3-vertex. If $v$ is incident with a 3-face, then $v$ is incident with two $5^+$-faces by the assumption of $H$. So, by r2.1, $ch'=-1+2\times \frac{1}{2}=0$. If $v$ is not incident with any 3-face, then  $v$ is incident with three $4^+$-faces. By r2.2 it has that $ch'(v)=-1+3\times \frac{1}{3}=0$.

Let $v$ be a 4-vertex. By  r1 to r4, we have $ch'(v)=ch(v)=0$.

Let $v$ be a 5-vertex.
When $v$ is not incident with any 3-face, we have that $v$ is incident with at most five 4-faces. So, by r4, $ch'(v)\geq 1-5\times \frac{1}{5}$=0.
When $v$ is incident with 3-faces, let $n'_3 (\geq 1)$ and $n'_4$ be the number of 3-faces and 4-faces incident with $v$, respectively. If $n'_3=1$, then $n'_4\leq 2$ (otherwise there is a 4-face that is adjacent to a 3-face, and a contradiction). So, by r3 and r4, $ch'(v)\geq 1-\frac{1}{2}-2\times\frac{1}{5}$=$\frac{1}{10}$. If $n'_3$=2, then $n'_4=0$. Therefore, by r3 we have $ch'(v)\geq 1-2\times \frac{1}{2}$=$0$.


Let $v$ be a 6-vertex. Denote by $n_3$ and $n_4$ the number of 3-faces and 4-faces incident with $v$, respectively. Obviously, $n_3+n_4\leq 3$ by the assumption of $H$.
If $v$ is not adjacent to a 2-vertex, then it is clear that $ch'(v)\geq 6-4-\frac{1}{2}-\frac{1}{2}-\frac{1}{2}=\frac{1}{2}$. If $v$ is adjacent to a 2-vertex that is incident with a 3-face, then $v$ is adjacent to only one 2-vertex by Lemma \ref{lemma2}. Therefore,  $ch'(v)\geq 6-4-\frac{1}{2}-\frac{1}{2}-\frac{1}{2}-\frac{1}{2}=0$ by r1, r3 and r4. So, in what follows we assume that $v$ is adjacent to at least two 2-vertices, and no such 2-vertex is incident with a 3-face. Clearly, in this case $v$ gives its adjacent 2-neighbors at most 1 by r1.

If $n_3+n_4\leq 2$, then $ch'(v)\geq 6-4-1-\frac{1}{2}-\frac{1}{2}=0$ by r1, r3 and r4.
If $n_3+n_4=3$, then denote by $F_1, F_2,F_3$ the three $5^+$-faces incident with $v$.
When $n_3\leq 1$, it has that $F_1,F_2,F_3$ are $6^+$-faces by the condition (2) of Theorem \ref{mainresult}.  Given that $v$ is adjacent to at least two 2-vertices, we have that $v$ is incident with at most one (3,5,3,6)-face and exactly two (2,6,$3^+$,6)-faces, and by Lemmas  \ref{chargeto6-vertex1}, \ref{chargeto6-vertex2} and \ref{chargeto6-vertex}  $v$ can receive at least $\frac{1}{6}+\frac{1}{8}+\frac{1}{8}=\frac{10}{24}$ from $F_1$, $F_2$ and $F_3$. Therefore, $ch'(v)\geq 6-4-1-\frac{1}{2}-\frac{5}{12}-\frac{5}{12}+\frac{10}{24}=\frac{1}{12}$ by r1, r3 and r4. When $n_3=3$, we have that $v$ is not incident with any 2-vertex. Therefore,  $ch'(v)\geq 6-4-\frac{1}{2}-\frac{1}{2}-\frac{1}{2}=\frac{1}{2}$ by r3.

Now, we consider the case of   $n_3=2$ and $n_4$=1. Obviously, in this case $v$ is incident with only one 2-vertex  by Lemma \ref{lemma3}, say $v_1$. Denote by $S_2,S_3$ the two $3$-faces incident with $v$, $S_1$ the 4-face incident with $v$,  and $F_1,F_2,F_3$ the other three faces incident with $v$; See Figure \ref{figureadd}.   Obviously, $v_1$ is incident with $S_1$, and $F_1,F_2$ are $6^+$-faces by the condition (2) of Theorem \ref{mainresult}.

By Lemma \ref{lemma3}, $d_H(v_2)\geq 3$. First, when $d_H(v_2)\geq 4$,  $v$ can receive at least $\frac{1}{6}+\frac{1}{8}=\frac{7}{24}$ from $F_1$ and $F_2$ by Lemmas  \ref{chargeto6-vertex2} and \ref{chargeto6-vertex}. Therefore, $ch'(v)\geq 6-4-1-\frac{1}{2}-\frac{1}{2}-\frac{1}{4}+\frac{7}{24}=\frac{1}{24}$ by r1, r3 and r4. Second, when $d_H(v_2)=3$. Consider  $F_1$: if $d_H(v_6)\neq 3$, then $F_1$  gives $v$ at least $\frac{2}{9}$  by Lemmas \ref{chargeto6-vertex2} and \ref{chargeto6-vertex}; if $d_H(v_6)=3$,  then $F_1$  gives $v$ at least $\frac{1}{6}$  by Lemmas \ref{chargeto6-vertex2} and \ref{chargeto6-vertex}.
Consider  $F_2$: if $F_2$ is a 6-face, then $F_2$ gives $v$ at least $\frac{2}{9}$ (or $\frac{1}{4}$) when $d_H(v_3)=3$ (or $d_H(v_3)\geq 4$) by Lemma \ref{chargeto6-vertex2}; if $F_2$ is a 7-face, then $F_2$ is incident with at most one
2-vertices that are incident with a 3-face. Otherwise, there exists a 4-cycle $S_1$ that is adjacent to a 5-cycle. Therefore, $F_2$ gives $v$ at least $\frac{7-4-1-\frac{1}{2}-\frac{1}{2}}{4}=\frac{1}{4}$. If $F_2$ is a $8^+$-face, then $F_2$ gives $v$ at least $\frac{k-4-\ell\times \frac{1}{6}-\ell_2-(\ell_3-1)\times \frac{1}{2}- \frac{1}{3}}{k-\ell_2-\ell_3-\lceil\frac{\ell}{2}\rceil}$ $\geq \frac{1}{3}$ (note that in this case $\ell_2\leq \lceil\frac{k-4}{2}\rceil$), where $\ell$, $\ell_2$ and $\ell_3$ are the number of $(4^-,4^-,5^+)$-faces adjacent to $F_2$, 2-vertices incident with $F_2$ and 3-vertices incident with $F_2$, respectively.
Therefore,  $v$ can receive at least $\min \{\frac{1}{4}+\frac{1}{6},\frac{2}{9}+\frac{2}{9}\}=\frac{5}{12}$ from $F_1$ and $F_2$ when $d_H(v_3)\geq 4$ or $d_H(v_6)\geq 4$, and  receive at least $\frac{1}{6}+\frac{2}{9}=\frac{7}{18}$ from $F_1$ and $F_2$ when $d_H(v_3)=3$ and $d_H(v_6)=3$. In addition, when  $d_H(v_3)=3$ and $d_H(v_6)=3$, it has that $d_H(v_4)\geq 5$ or $d_H(v_5)\geq 5$ by Lemma \ref{lemma1} (1). In this case, it is easy to see that $F_3$ can give $v$ at least $\frac{1}{8}$. Hence, $v$ can receive at least $\frac{5}{12}$ from $F_1$, $F_2$ and $F_3$, and  $ch'(v)\geq 6-4-1-\frac{1}{2}-\frac{1}{2}-\frac{5}{12}+\frac{5}{12}=0$  by r1, r3 and r4.

\begin{figure}[H]
\centering
  \includegraphics[width=80pt]{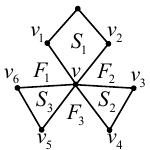}\hspace{0.3cm}\\
   \caption{An Illustration}\label{figureadd}
\end{figure}

\subsection{Final charge of faces}

By Lemma \ref{6faces}, we are sufficient to check that all $3$-faces and $4$-faces have nonnegative charge.

Let $f$ be a 3-face. Since $H$ contains no $(4,4,4)$-face by Lemma \ref{lemma1} (3), it follows that $f$ is incident with at most two $4^-$-vertices by Lemma \ref{lemma1} (1).  Denote by $\ell$ the number of $4^-$-vertices incident with $v$. If $\ell=0$, then by r3.1 $ch'(f)=3-4+3\times \frac{1}{3}=0$; if $\ell=1$, then by r3.2 $ch'(f)=3-4+2\times \frac{1}{2}=0$; if $\ell=2$, then by r3.3 $ch'(f)=3-4+\frac{1}{2}+3\times \frac{1}{6}=0$.

Let $f$ be a 4-face. Obviously, $f$ is not adjacent to a 3-face. If $f$ is not incident with any 2-vertex, then when $f$ is not incident with any 3-vertex, $ch'(f)=ch(f)=0$; when $f$ is  incident with a 3-vertex, since $f$ is not (3,5,3,5)-face by Lemma \ref{lemma4}, it has that $f$ is incident with at least one 6-vertex. Thus, if $f$ is incident with one 3-vertex, then $ch'(f)\geq 3\times \frac{2}{15}-\frac{1}{3}=\frac{1}{15}$ or $ch'(f)\geq 2\times \frac{1}{6}-\frac{1}{3}=0$ by r4.2; if $f$ is incident with two 3-vertices, then $ch'(f)\geq \min \{\frac{7}{15}+\frac{1}{5}-2\times \frac{1}{3}, \frac{1}{3}+\frac{1}{3}-2\times \frac{1}{3}\}=0$ by  r4.2.

In addition, if $f$ is incident with a 2-vertex (note that $f$ is incident with  at most one 2-vertex by Lemma \ref{lemma2}), then when $f$ is incident with a 3-vertex, it has that $ch'(f)\geq 2\times \frac{5}{12}-\frac{1}{2}-\frac{1}{3}=0$ by r4.1; when  $f$ is not incident with a 3-vertex, it follows that $ch'(f)\geq 2\times \frac{1}{4}-\frac{1}{2}=0$ by r4.1.

This completes the proof of Theorem \ref{mainresult}.

\section{Conclusion}

For any planar graph with maximum degree 6, it has been proved that in the case of not containing 4-cycles, it is 7-totally-colorable\cite{Shen2009}.
In this paper, we improve this result by showing that it is still 7-totally-colorable when it contains 4-cycles except some special ones.

For future work, we would like to refine our result further. In fact, the graphs we have studied belong to the so-called type 1  planar graphs,  i.e. graphs with total chromatic number $\Delta+1$. In \cite{Kowalik2008}, Kowalik et al. proved that planar graphs with maximum degree $\Delta \geq 9$ are $(\Delta+1)$-totally-colorable.  Many people conjecture that  planar graphs with $4\leq \Delta\leq 8$ are also type 1 graphs \cite{Shen2009}, since no counterexample has been found. For the case of $\Delta=3$, there do exist planar graphs without a $4$-total-coloring, for example $K_4$ or the graphs shown in Figure \ref{figure3}. However, the characteristics of such graphs are unknown. For the graphs we have studied  an interesting phenomena is that every non 4-totally-colorable planar graph with $\Delta=3$ contains a 4-cycle.  Therefore, we propose the following conjecture.

\begin{figure}[H]
\centering
  \includegraphics[width=67pt]{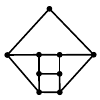}\hspace{0.5cm}
  \includegraphics[width=55pt]{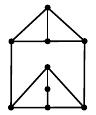}\hspace{0.7cm}
   \includegraphics[width=70pt]{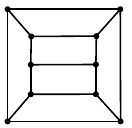}
  \caption{Non 4-totally-colorable planar graphs of maximum degree 3 }\label{figure3}
\end{figure}

{\bf{Conjecture 3.1.} }
Let $G$ be a planar graph of maximum degree 3. If $G$ contains no 4-cycle, then $G$ is 4-totally-colorable.


\end{document}